\DeclareUnicodeCharacter{2212}{-}
\documentclass[12pt]{amsart}
\usepackage[utf8]{inputenc}
\usepackage[T1]{fontenc}
\usepackage{amsmath}
\usepackage{amsfonts}
\usepackage{amssymb,amsthm}
\usepackage{graphicx}
\usepackage{mathrsfs}
\usepackage{hyperref}
\usepackage{color}
\usepackage[dvipsnames]{xcolor}
\usepackage{etoolbox}
\usepackage{comment}
\usepackage{appendix}
\usepackage{enumerate}
\usepackage{enumitem}
\newtheorem{thm}{Theorem}[section]

\newtheorem{cor}[thm]{Corollary}

\newtheorem{rem}[thm]{Remark}

\def\qq#1{\qquad \mbox{#1}\quad}
\def\q#1{\quad \mbox{#1}\ }

\newcommand{\be}{\beta}

\newcommand{\e}{\varepsilon }

\newcommand{\g}{\gamma }

\newcommand{\na}{\nabla }

\newcommand{\Om}{\Omega }
\newcommand{\Omb}{\overline{\Om}}
\newcommand{\p}{\partial }
\newcommand{\s}{\sigma}
\newcommand{\te}{\theta}

\title[Uniform estimates for Caratheodory nonlinearities]{Uniform estimates for elliptic equations with Caratheodory nonlinearities at the interior and on the boundary}

\author[E.~Antonio]{Edgar Antonio}
\address[E.~Antonio]{Universidad Autónoma de Guerrero, Guerrero, México}
\email{eaam020713@gmail.com}
\author[M.~Arciga-Alejandre]{Martín P. Árciga-Alejandre}
\address[M.~Arciga-Alejandre]{Universidad Autónoma de Guerrero, Guerrero, México}
\email{mpargica@uagro.mx}
\author[R.~Pardo]{Rosa Pardo}
\address[R.~Pardo]{Universidad Complutense de Madrid, Madrid, Spain}
\email{rpardo@ucm.es}
\author[J.~Sanchez-Ortiz]{Jorge Sánchez Ortiz}
\address[J.~Sanchez-Ortiz]{Universidad Autónoma de Guerrero, Guerrero, México}
\email{jsanchez@uagro.mx}

\thanks{The third author is supported by grants  PID2022-137074NB-I00,  MICINN,  Spain, and by UCM, Spain,  Grupo 920894.}

\date{}
\begin{document}

\begin{abstract}
We establish an explicit uniform a priori estimate for weak solutions to slightly subcritical elliptic problems with nonlinearities simultaneously at the interior and  on the boundary. Our explicit $L^{\infty}(\Om )$ a priori estimates are in terms of powers of their $H^{1}(\Om )$ norms. To prove our result, we combine a De Giorgi-Nash-Moser's iteration scheme together with elliptic regularity and the Gagliardo-Nirenberg's interpolation inequality. \\
\end{abstract}
\maketitle



\section{Introduction}
Let us consider the boundary value problem with nonlinear boundary conditions
\begin{equation}\label{E0.1}
\left \{
\begin{aligned}
-\Delta u +u=& f(x,u), \; \; \; x\in \Om , \\
\frac{\p u}{\p \eta  }=& f_{B}(x,u), \;x\in \p \Om ,
\end{aligned}
\right . 
\end{equation}
where $\Om \subset \mathbb{R}^{N}$, ($N>2$), is an open, connected, bounded domain with $C^{2}$ boundary, $\p /\p \eta  =\eta \cdot\na$ is the (unit) outer normal derivative, and the  functions $f:\Om\times \mathbb{R}\rightarrow \mathbb{R}$, and  $f_{B}:\p \Om \times \mathbb{R} \rightarrow \mathbb{R}$, are both slightly subcritical  Carath\'eodory functions. We give in \ref{f1}-\ref{f2} and \ref{fB1}-\ref{fB2}  the precise statement of the hypothesis on the nonlinearities at the interior, and on the boundary respectively.\\

Our goal is to establish  explicit $L^\infty (\Om )$  {\it a priori}  estimates for weak solutions to \eqref{E0.1}, in terms of powers of their  $H^1(\Om)$ norms  (see Theorem \ref{th:CotasLinf}).  Our estimates are valid to positive solutions and to changing sign solutions. Consequently, any sequence of solutions to \eqref{E0.1}, uniformly bounded in the $H^1(\Om)$ norm, is also uniformly bounded in the $L^\infty (\Om)$ norm. \\

Our techniques are based in an iterative process due to Moser, in the elliptic regularity theory, and in the Gagliardo--Nirenberg interpolation inequalities.\\

For the Homogeneous Dirichlet boundary conditions, by a Moser's type procedure, it is well known that weak solutions to a subcritical or even critical elliptic problem are in $L^q(\Omega )$ for all $1<q<\infty $  (see \cite[Lemma 1]{Moser},  see also \cite[Section 2.2]{Drabek_Kufner_Nicolosi}, \cite[Lemma B.3]{Struwe}. Moreover, by elliptic regularity, the solutions are in $L^{\infty}(\Om)$.\\

Moser's results can be extended to the case of nonlinear boundary conditions, 
and also to a general quasilinear problem, which includes in particular \eqref{E0.1}, see, for instance,  \cite[Theorem 3.1]{Marino_Winkert_NonlAnal_2019}. In \cite{Marino_Winkert_NonlAnal_2019} the authors state that weak solutions to some quasilinear problem are in $L^{\infty}(\Om)\cap L^{\infty}(\p \Om)$. By elliptic regularity, weak solutions to \eqref{E0.1} are in fact more regular, and in particular, they are uniformly continuous functions. Indeed, the elliptic regularity theory, applied to weak solutions of a subcritical or even critical problem implies that they are in $C(\Omb)$, see estimate \eqref{estim:C:nu} in Theorem \ref{th:reg}. So, in that case, 
\begin{equation} \label{Continuity}
     \|u\|_{L^{\infty}(\p\Om)}\le \|u\|_{C(\Omb)}=\|u\|_{L^{\infty}(\Om)}.
\end{equation}
In addition to improving the regularity, uniform $L^\infty (\Om )$ a priori bounds, joint with Leray-Schauder degree theory, allow us to investigate results on the existence of solutions. Leray-Schauder degree theory (see \cite{Leray_Schauder}) is a powerful tool to obtain results on the existence of solutions. To define the degree, uniform $L^\infty  (\Om)$ a priori bounds are needed. Results on uniform $L^\infty (\Om )$ a priori bounds by Gidas and Spruck can be found in \cite{Gidas_Spruck_bd} and by de Figueiredo, Lions, and Nussbaum in \cite{deFigueiredo_Lions_Nussbaum_1982}, see also \cite{Castro_Pardo_RMC_2015}   for a slightly subcritical Dirichlet problem, \cite{Castro_Pardo_DCDS_2017} for a non-convex domain, \cite{Mavinga_Pardo_JMAA} for an elliptic system,\cite{Damascelli_Pardo} for the $p$-laplacian case, and \cite{Pardo_Int_2019_I, Pardo_Int_2019_II}  for a review.  \\
 
The type of $L^\infty (\Om )$ estimates given by \eqref{u:bd:1} are known for slightly subcritical nonlinearities in the  homogeneous Dirichlet problem with the laplacian operator, see \cite[Theorem 1.5]{Pardo_JFPTA_2023}, with the $p$-Laplacian operator, see \cite[Theorem 1.6]{Pardo_RACSAM}, and also with a linear problem at the interior joint with nonlinear boundary conditions on the boundary of power type, see \cite{Chhetri_Mavinga_Pardo_PAMS}. \\

In this paper, we analyze the combined effect of both nonlinearities simultaneously. Our main result establishes the explicit  estimates provided by Theorem \ref{th:CotasLinf}, where both nonlinearities in the interior and on the boundary are slightly subcritical,  not necessarily of power type. \\

This paper is organized in the following way.  Section \ref{sec:main} contains the statement of our main result, Theorem \ref{th:CotasLinf}. The proof of Theorem \ref{th:CotasLinf} is achieved in Section \ref{sec:proof}. By the sake of completeness, we include two appendix; in Appendix \ref{ApB}, we recall the regularity of weak solution to the linear problem with non homogeneous data both at the interior and on the boundary, see Theorem \ref{ApB1}; Appendix  \ref{sec:prelim} deals with further regularity of weak solutions to \eqref{E0.1}, see Theorem \ref{th:reg}.


\section{Our main result}
\label{sec:main}
In this section we state our main result. \\

For $p>1$, we define the trace operator: 
\begin{equation*} 
\Gamma :W^{1,p}(\Om)\rightarrow L^{p}(\p \Om ),
\end{equation*}
in the following way
\medskip

\begin{enumerate}
    \item $\Gamma u= u|_{\p \Om }$ \,  if \, $u\in W^{1,p}(\Om )\cap C(\Omb)$,
\medskip  
    \item $\|\Gamma u\|_{L^{p}(\p \Om )}\leq C\|u\|_{W^{1,p}(\Om )}$,
\medskip
\end{enumerate}
where $C=C(p,|\Om|)$ is a constant and $\p \Om $ is $C^{1}$. 
Since the surjectivity and the continuity of the trace operator, we get 
\begin{equation*} 
\Gamma :W^{1,p}(\Om)\rightarrow W^{1-\frac{1}{p},p}(\p \Om ) \hookrightarrow L^{q}(\p \Om ), \quad \text{for} \quad  1\le q\le \frac{(N-1)p}{N-p},
\end{equation*}
and 
\begin{equation*}
\|\Gamma u \|_{L^{q}(\p \Om )}\le C\|u\| _{W^{1,p}(\Om)}, \quad \text{for some} \quad C>0,
\end{equation*}
this operator is continuous for $1\le q\le \frac{(N-1)p}{N-p}$, and compact for  $1\le q< \frac{(N-1)p}{N-p}$ (see  \cite[Theorem 6.4.1]{Kufner_Fucik} and \cite[Lemma 9.9]{Brezis}). \\

Throughout this paper, we use the Sobolev embedding
\begin{equation} \label{Encaje:Sobolev}
    H^{1}(\Om )\hookrightarrow L^{2^*}(\Om ),
\end{equation}
and  the continuity of the trace operator 
\begin{equation*} 
H^{1}(\Om )\hookrightarrow L^{2_*}(\p \Om ) ,
\end{equation*}
where 
\begin{equation}\label{def:*}
2^*:=\frac{2N}{N-2} \qquad \text{and} \qquad 2_*:=\frac{2(N-1)}{N-2},
\end{equation}
are the  critical Sobolev exponent and the critical exponent in the sense of the trace, respectively.\\

For $1< p,\ p_B\le \infty$, we will denote
\begin{equation}\label{def:N/p&N/pB}
2^*_{N/p}:=\frac{2^*}{p'}=2^*\Big( 1-\frac{1}{p}\Big) \quad \text{and}\quad 2_{*,N/ p_{B}}:=\frac{2_*}{p_{B}'}=2_*\Big( 1-\frac{1}{p_B}\Big) ,
\end{equation}
where $p'$ is the conjugate exponent of $p$, that is $\frac{1}{p}+\frac{1}{p'}=1$. And also
\begin{equation*}
p^*:=\frac{Np}{N-p}, \;  p_*:=\frac{(N-1)p}{N-p}=\frac{(N-1)p^*}{N}, \q{for} 1\le p< N,
\end{equation*}
will denote the  critical Sobolev exponent  and the critical exponent in the sense of the trace, respectively.
\medskip

Given $f:\Om\times \mathbb{R}\rightarrow \mathbb{R}$, we will assume   the following hypothesis on the nonlinearity at the interior:
\smallskip
\begin{enumerate}[label=\textbf{(f\arabic*)}]
\item
\label{f1}
$f$ is a {\it Carath\'eo\-dory} function:
\begin{enumerate}
\item $f(\cdot , t)$ is measurable for each  $t\in \mathbb{R}$;
\item $f(x, \cdot )$, is continuous for each $x \in \Om$.
\end{enumerate}
\item 
\label{f2} $f$ is {\it slightly subcritical (at infinity)}, that is: 
\begin{equation}\label{E0.2}
|f(x,t)| \leq |a(x)|\tilde{f}(|t|),
\end{equation}
with $a(x)\in L^{r}(\Om )$ for $r>\frac{N}{2}$, $\tilde{f}:[0,+\infty )\rightarrow [0,+\infty )$
is continuous, and such that
\begin{equation}\label{E0.3}
\lim _{ t\rightarrow +\infty }\frac{\tilde{f}(t)}{ t ^{2^*_{N/r}-1}}=0.
\end{equation}
\end{enumerate}

\medskip

Likewise, given $f_{B}:\p \Om \times \mathbb{R} \rightarrow \mathbb{R}$, we will assume the following hypothesis for the nonlinearity on the boundary:
\begin{enumerate}[label=\textbf{(f$_B$\arabic*)}]
\item
\label{fB1}
$f_{B}$ is a {\it Carath\'eo\-dory} function:
\begin{enumerate}
\item $f_{B}(\cdot , t)$ is measurable for each  $t\in \mathbb{R}$;
\item $f_{B}(x, \cdot )$ is continuous for each $x \in \p \Om$.
\end{enumerate}
\item
\label{fB2}
$f_{B}$ is {\it slightly subcritical (at infinity)}, that is:
\begin{equation}\label{E0.4}
|f_{B}(x,t)| \leq |a_{B}(x)|\,  \tilde{f}_{B}(|t|),
\end{equation}
with $a_B(x)\in L^{r_B}(\p \Om )$ for $ r_{B}>N-1 $,  and $\tilde{f}_B:[0,+\infty )\rightarrow [0,+\infty )$ is 
continuous, such that
\begin{equation}\label{E0.5}
\lim _{ t\rightarrow +\infty }\frac{\tilde{f}_B(t)}{ t ^{2_{*,N/ r_{B}}-1}}=0.
\end{equation}
\end{enumerate}

\medskip

We say that $u\in H^{1}(\Om )$ is a {\it weak solution} to \eqref{E0.1} if $f(\cdot, u)\in L^{(2^*)'}(\Om )$, and $f_{B}(\cdot, u)\in L^{(2_*)'}(\p \Om )$ are such that for all $\psi\in H^{1}(\Om ),$
\begin{equation*}
\int_{\Om}\na u \na \psi \; dx + \int_{\Om}u\psi \; dx = \int_{\Om}f(x,u)\psi\,dx+ \int_{\p \Om }f_{B}(x,u)\psi\, dS \,,
\end{equation*}
being $(2^*)'
=\frac{2N}{N+2}$ and $(2_*)'
=\frac{2(N-1)}{N}$ the conjugate exponents of $2^*$ and $2_*$, respectively.

\begin{rem}
(i) Let $u\in H^{1}(\Om )$. By Sobolev embeddings, for $f$ and $f_B$ slightly subcritical (satisfying {\rm\ref{f2}} and {\rm\ref{fB2}} respectively), we have
\begin{align*}
\tilde{f}(u)\in L^{\frac{2^*}{2^*_{N/r}-1}}(\Om ), \qquad\text{where}& \qquad \frac{2^*_{N/r}-1}{2^*}=\frac{1}{2}+\frac{1}{N}-\frac{1}{r},\\
\tilde{f}_{B}(u)\in L^{\frac{2_*}{2_{*,N/ r_{B}}-1}}(\p \Om ), \qquad\text{where}& \qquad \frac{2_{*,N/ r_{B}}-1}{2_*}=\frac{N}{2(N-1)}-\frac{1}{r_B}.
\end{align*}
Hence, 
\begin{equation*}
f(\cdot ,u)\in L^{(2^*)'}(\Om ) \quad \text{and} \quad  f_{B}(\cdot ,u)\in L^{(2_*)'}(\p \Om ),
\end{equation*}
(ii) We can allways choose $\tilde{f}$ and $\tilde{f}_{B}$ such that $\tilde{f}(t)>0$ and $\tilde{f}_{B}(t)>0$ for $s>0$. Note that redefining both functions, $\tilde{f}(t )$ and $\tilde{f}_{B}(t)$, as $\max_{[0,t]}\tilde{f}$ and $\max_{[0,t]}\tilde{f}_{B}$, respectively, we can always choose $\tilde{f}$ and $\tilde{f}_{B}$ as non decreasing functions  for $s>0$.
\end{rem}
\smallskip

Now, we will define two new functions, $h$ and $h_B$, which will be essential for the statement of our main result. 
Let us define
\begin{equation}\label{def:h&hB}
h(s):=\frac{s^{2^*_{N/r}-1}}{\tilde{f}(s)} \quad \text{and} \quad h_{B}(s):=\frac{s^{2_{*,N/ r_{B}}-1}}{\tilde{f}_{B}(s) } \quad \text{for} \ s>0,
\end{equation}
Since the nonlinearities $f$ and $f_{B}$ are both slightly subcritical, in other words, satisfy conditions \eqref{E0.2}-\eqref{E0.3} and \eqref{E0.4}-\eqref{E0.5}, respectively, then
\begin{equation}\label{h:infty}
h(s)\rightarrow \infty    \qq{and} h_{B}(s)\rightarrow \infty \qq{as}s\rightarrow \infty.
\end{equation}
\bigskip
 
Our main result is given in the following theorem, which applies for weak solutions to \eqref{E0.1} with slightly subcritical Carathéodory nonlinearities. Our estimates apply to positive or even changing sign solutions.\\

We will denote as $a_M$ the maximum of the corresponding norms of $a\in L^{r}(\Om )$ and of $a_B\in L^{r_B}(\p \Om )$, that is
\begin{equation}\label{def:aM}
a_M:=\max \lbrace \| a\|_{L^{r}(\Om )} , \| a_B\|_{L^{r_B}(\p \Om )} \rbrace .
\end{equation}

Let $u$ be a solution to \eqref{E0.1}. Let $h_{m}$ be defined  as the minimum of $h$ and a certain power of $h_B$, specifically
\begin{equation}\label{def:hm}
h_{m}(s):=\min \left\{ h(s),\, h_{B}^{\frac{2^*_{N/r}-1}{2_{*,N/r_B}-1}}(s)\right\} , 
\end{equation}
with $h$ and $h_B$ defined in \eqref{def:h&hB}.\\

The following Theorem contains our estimates of $h_{m}(\| u\|_{L^\infty (\Om ) })$ in terms of their $H^{1}(\Om)$ norms.

\begin{thm} \label{th:CotasLinf}
Let $f:\Om \times \mathbb{R}\rightarrow \mathbb{R}$ and $f_{B}=\p \Om \times \mathbb{R}\rightarrow \mathbb{R}$ be Carath\'eodory functions, satisfying {\rm \ref{f1}--\ref{f2}} and {\rm \ref{fB1}--\ref{fB2}}, respectively. Let $u\in H^{\, 1}(\Omega )$ be an arbitrary weak solution to \eqref{E0.1}.\\ 

Then, for all $\, \e >0$, there exist $C_{\e }>0$ depending of $\e $,  $N$, $|\Om | $ and $|\p \Om | $, but independent of $u$, such that
\begin{equation}\label{u:bd:1}
h_{m}(\| u\|_{L^\infty (\Om ) })\leq C_{\e } a_{M} ^{A+\e } \left( 1+\| u\|_{H^{1}(\Om )}^{(2^*_{N/r}-2)(A+\e )}\right),
\end{equation} 
where $h_{m}$ is defined by \eqref{def:hm}, $a_{M}$  by \eqref{def:aM}, and
\begin{equation} \label{def:A}
A:=\begin{cases}
\dfrac{\dfrac{1}{2}-\dfrac{N-r}{Nr}}{\dfrac{1}{2}-\dfrac{N-1}{Nr_{B}}} &\text{ if }
\begin{cases}
\text{either }   r\ge N,\\[.2cm]
\text{or } N/2<r< N  \text{ and } r^*\ge \frac{Nr_B}{N-1},
\end{cases}
\\[1cm]
1&\text{ if} \quad N/2<r< N \text{ and } r^*\le \frac{Nr_B}{N-1}.
\end{cases}
\end{equation}
\end{thm}
\bigskip

\begin{rem}
Since \eqref{Continuity}, in fact
\begin{equation*}
h_{m}(\| u\|_{C(\Omb ) })\leq C_{\e } a_{M} ^{A+\e } \left( 1+\| u\|_{H^{1}(\Om )}^{(2^*_{N/r}-2)(A+\e )}\right).
\end{equation*} 
\end{rem}

\begin{rem}
From the definitions of $h$ and $h_{B}$ given in \eqref{def:h&hB}, we note that 
\begin{equation*}
h(s) =\frac{s^{2^*_{N/r}-1}}{\tilde{f}(s)} \qq{and}   h_{B}^{\frac{2^*_{N/r}-1}{2_{*,N/r_B}-1}}(s)  
=\left(\frac{s^{2_{*,N/ r_{B}}-1}}{\tilde{f}_{B}(s) }\right)^{\frac{2^*_{N/r}-1}{2_{*,N/r}-1}}.
\end{equation*}
Thus, 
\begin{equation*} 
h_{m}(s)=\min \left\{ \frac{s^{2^*_{N/r}-1}}{\tilde{f}(s)}, \frac{s^{2^*_{N/r}-1}}{\tilde{f}_{B}^{\frac{2^*_{N/r}-1}{2_{*,N/r}-1}}(s)}\right\}. 
\end{equation*} 
\end{rem}
\bigskip

\section{\texorpdfstring{$L^\infty (\Omega )$}{}  a priori explicit estimates}
\label{sec:proof}

In this section, assuming that $f$ and $f_{B}$ are Carathéodory functions satisfying slightly subcritical growth conditions, we prove our main result.

Our method combines elliptic regularity with the Gagliardo--Nirenberg interpolation inequality. Let $u$ be an arbitrary solution to \eqref{E0.1}. First, we find estimates of the nonlinearities in terms of products of the $H^{1}(\Om)$-norm  of $u$ and their $L^{\infty}(\Om)$-norm. With it, using elliptic regularity (see Theorem \eqref{Th:RegLineal}), we obtain estimates of the $W^{1,m}(\Om)$-norm, with $m>N$, of the solutions to \eqref{E0.1}. Finally, applying the Gagliardo--Nirenberg interpolation inequality, (see \cite{Nirenberg_1959}), we obtain an explicit estimate of the $L^{\infty}(\Om)$-norm of $u$  in terms of the $H^{1} (\Om)$ norm of $u$.

\begin{proof}[Proof of Theorem \ref{th:CotasLinf}]
Let $u \in H^{1}(\Om )$ be a weak solution to \eqref{E0.1}. By  Theorem \ref{th:reg}, $u \in H^{1}(\Om )\cap L^\infty  (\Om )$.\\

Firstly, we will estimate both nonlinearities (the interior  and the boundary nonlinearities)  in terms of the $H^{1}(\Om)$-norm and the $L^\infty  (\Om )$-norm of $u$.\\

{\bf Step 1.} {\it $W^{1,m}(\Om )$ estimates for $m>N$.}
\bigskip

By hypothesis,  $\tilde{f}$ and $\tilde{f_B}$ are both increasing. By \eqref{Continuity}we denote 
\begin{align}\label{def:M:MB}
M&:= \tilde{f}(\| u\|_{L^\infty (\Omega )})=\max _{\left[ 0, \| u\|_{L^\infty (\Omega )}\right] } \tilde{f},\\ 
M_{B}&:= \tilde{f_{B}}(\| u\|_{ L^\infty (\Omega )})=\max _{\left[ 0, \| u\|_{L^{ \infty }(\Om )}\right] }\tilde{f_{B}}.\nonumber 
\end{align}

Along this proof, we will use the obvious fact that for any $\g >0$, there exist two constants $C_1$ and $C_2$, only dependent on $\g$, such that 
\begin{equation} \label{des:1}
C_1(1+x^{\g })\leq (1+x)^{\g }\leq C_2(1+x^{\g }), \qquad \text{for all} \quad x\geq 0.
\end{equation}
From now on, all throughout this proof, $C$ denotes several constants independent of $u.$
\medskip

By the growth condition \eqref{E0.2} and the definition of $M$ \eqref{def:M:MB},  we have that
\begin{align}\label{S3-1}
\int_{\Om}|f(\cdot ,u)| ^{q}dx
&\leq \int_{\Om}|a(x)|^{q}\tilde{f}(u)^{q-t+t}\,dx \nonumber \\
&\leq C M^{q-t} \int_{\Om}|a(x)|^{q}\tilde{f}(u)^{t}\,dx,
\end{align}  
for any $t<q$, and any 
\begin{equation}\label{interv:q}
q\in \left(\frac{N}{2},\min \lbrace r,N \rbrace \right).    
\end{equation} 
Using the H\"older's inequality, for all $1< s< \infty$, we can write
\begin{equation}\label{ts'}
\int_{\Om}|a(x)|^{q}\tilde{f}(u)^{t}\,dx  \leq \left( \int_{\Om}|a(x)|^{qs} dx\right) ^{\frac{1}{s}} \left(\int_{\Om}\tilde{f}(u)^{ts'}dx \right)^{\frac{1}{s'}} ,
\end{equation} 
where   $s'$ is such that $\frac{1}{s}+\frac{1}{s'}=1.$
Choosing $s$ and $t<q$, so that  $qs=r$  and $ts'=\frac{2^*}{2^*_{N/r}-1}$, thus,
\begin{align}\label{def:t}
&t:=\frac{2^*}{2^*_{N/r}-1}\left( 1-\frac{q}{r}\right) <q \\
&\iff \frac{1}{q}-\frac{1}{r}  < \frac{2^*_{N/r}-1}{2^*}=1-\frac{1}{r}-\frac{1}{2}+\frac{1}{N} \nonumber \\
  & \iff q >\frac{2N}{N+2}, \nonumber \quad \checkmark
\end{align}
since $q>\frac{N}{2}>\frac{2N}{N+2}$. 
\medskip

\noindent On the other hand, by subcriticality, see  \eqref{E0.3}, and Sobolev embeddings, see \eqref{Encaje:Sobolev},
\begin{align}\label{S3-2}
 \int_{\Om} |\tilde{f}(u)|^{\frac{2^*}{2^*_{N/r}-1}}dx &\leq C\int_{\Om}\left( 1+|u| ^{2^*} \right) \,dx \nonumber \\
 &\leq C\left(1 +\| u\|^{2^*}_{L^{2^*}(\Om )} \right) \,  \\
\label{S3-2b}
  &\leq C\left(1 +\| u\|^{2^*}_{H^{1}(\Om )} \right) \, . 
\end{align}

Finally, substituting \eqref{S3-2b} in the second factor on the RHS of \eqref{ts'},  this result in \eqref{S3-1} and since $1/(qs')=1/q-1/r,$ we get
\begin{equation}\label{Lq:norm}
\left(\int_{\Om} |f(\cdot ,u)|^{q}dx \right) ^{\frac{1}{q}}\leq C M^{1-\frac{t}{q}} \| a\|_{L^{r}(\Om )} \left( 1 +\| u\|_{H^{1}(\Om )} ^{2^*\left( \frac{1}{q}-\frac{1}{r}\right)} \right) . 
\end{equation}\\

Likewise, by the condition \eqref{E0.4}  and the subcriticality \eqref{E0.5}, we get
\begin{align}
\label{L:qB}
\int_{\p \Om }|f_{B}(\cdot ,u)|^{q_{B}}dS
&\leq \int_{\p\Om } |a_B(x)|^{q_B}\tilde{f}_B(u)^{q_B-t_B+t_B}\,dS \nonumber\\
&\leq C M_{B}^{q_B-t_B} \int_{\p\Om } |a_B(x)|^{q_B}\tilde{f}_B(u)^{t_B}\,dS,
\end{align}  
for any $t_B<q_B$, and any 
\begin{equation}\label{interv:qB}
q_B\in (N-1,r_B ).    
\end{equation} 
Using H\"older´s inequality, for all $1< s_{B}< \infty$, we obtain
\begin{align}\label{tBsB'}
&\int_{\p\Om } |a_B(x)| ^{q_B}\tilde{f}_B(u)^{t_B}\,dS \\
& \qquad \leq \left( \int_{\p\Om } |a_B(x)| ^{q_Bs_B} dS\right) ^{\frac{1}{s_B}} \left(\int_{\p\Om }\tilde{f}_B(u)^{t_Bs_B'}dS \right)^{\frac{1}{s_B'}}, \nonumber
\end{align}
where   $s_{B}'$ is such that  $\frac{1}{s_B}+\frac{1}{s_B'}=1$. Choosing, as before, $s_{B},\, t_{B}<q_{B}$, so that $q_Bs_B=r_B$,  and $t_Bs_B'=\frac{2_*}{2_{*,N/r_B}-1}$, thus, 
\begin{align}\label{def:tB}
&t_B :=\frac{2_*}{2_{*,N/r_B}-1}\left( 1-\frac{q_B}{r_B}\right) <q_B \\
& \iff 	\frac{1}{q_B}-\frac{1}{r_B}  <\frac{2_{*,N/r_B}-1}{2_*}=1-\frac{1}{r_B}-\frac{N-2}{2(N-1)}\nonumber \\
& \iff 	\frac{1}{q_B} < \frac{N}{2(N-1)}\ \iff q_B >\frac{2(N-1)}{N}, \nonumber \quad \checkmark
\end{align}
and the last inequality is satisfied since $q_B >N-1$ and  $N>2$.
\medskip

\noindent On the other hand, again by subcriticality, see \eqref{E0.4} and \eqref{E0.5}, 
\begin{align}\label{des:normftil}
\int_{\p \Om } |f_{B}(u)| ^{\frac{2_*}{2_{*,N/r_B}-1}}dx & \leq C\int_{\p\Om }\left( 1+| u| ^{2_*} \right) \,dS \nonumber\\
&\leq C\left(1 +\| u\|^{2_*}_{L^{2_*}(\p\Om )} \right) \\
\label{des:normftilb}
&\leq C\left(1 +\| u\|^{2_*}_{H^{1}(\Om )} \right) , 
\end{align}
then, substituting \eqref{des:normftilb} in the second factor on the RHS of \eqref{tBsB'}, this result in \eqref{L:qB}, and since $1/(q_Bs_B')=1/q_B-1/r_B,$ we get
\begin{align}\label{LqB:norm}
\left(\int_{\p \Om }|\tilde{f}_{B}(\cdot ,u)|^{q_{B}}dx \right) ^{\frac{1}{q_{B}}}&\leq C M_{B}^{1-\frac{t_B}{q_{B}}} \| a_B\|_{L^{r_B}(\p \Om )}\\
&\qquad \times\left( 1+\| u\|^{2_*\left( \frac{1}{q_B}-\frac{1}{r_B}\right) }_{H^{1}(\Om )} \right) . \nonumber
\end{align}

Now, using elliptic regularity, we estimate the norm $\|u\|_{W^{1,m}(\Om )}$ in terms of the corresponding norms of the nonlinearities,  see Theorem\ref{Th:RegLineal}, equation \eqref{W:1:m}. Specifically, using \eqref{Lq:norm} and \eqref{LqB:norm}, we obtain that
\begin{align}\label{RE}
\| u\|_{W^{1,m}(\Om )}&\leq C \left[ M^{1-\frac{t}{q}} \| a\|_{L^{r}(\Om )} \left( 1 +\| u\|_{H^{1}(\Om )} ^{2^*\left( \frac{1}{q}-\frac{1}{r}\right)}\right) \right. \nonumber\\
&  \left. \quad \quad +  M_{B}^{1-\frac{t_B}{q_{B}}} \| a_B\|_{L^{r_B}(\p \Om )}\left( 1+\| u\|^{2_*\left( \frac{1}{q_B}-\frac{1}{r_B}\right) }_{H^{1}(\Om )} \right) \right], 
\end{align}
where $m=\min \lbrace q^*, \frac{Nq_{B}}{N-1}\rbrace $ ($q^*:=\frac{Nq}{N-q}$), whenever $1\leq q<N$, see Theorem \ref{Th:RegLineal}. 

Fixing 
\begin{equation}\label{def:qB}
q_{B}:=\frac{(N-1)q^*}{N}
\implies m=q^*=\frac{Nq_{B}}{N-1}>N,
\end{equation}
(in the forthcoming Remark \ref{rem1}, we explain the necessity of the election for $q_B$), moreover, we have the following equivalences
\begin{equation}\label{equiv:q:qB}
q_{B}:=\frac{(N-1)q^*}{N}\iff \frac{2_*}{q_{B}}=\frac{2^*}{q^*}\iff 2_{*,N/q_B} =  2^*_{N/q}.
\end{equation}
Indeed, we only have to notice that, using the definitions \eqref{def:*}, \eqref{def:N/p&N/pB} and \eqref{def:qB}, we can conclude that
\begin{equation*}
2_{*,N/q_B} = 2_*-\frac{2_*}{q_B} = 2_*+\frac{2^*}{N}-\frac{2^*}{q} = 2^*-\frac{2^*}{q} = 2^*_{N/q}. 
\end{equation*}

With that election of $q_B$, we also need to restrict $q$ in order to satisfy \eqref{interv:qB}. Specifically
\begin{equation}\label{interv:q:2}
q\in \left(\frac{N}{2},\min \left\{ r,\frac{Nr_B}{N-1+r_B}\right\} \right). 
\end{equation}
\\
Indeed, note that, because of the definition of $q_B$, see \eqref{def:qB}, and their restriction, \eqref{interv:qB}, the following inequality have to be satisfied 
\begin{equation*} 
N-1<\frac{(N-1)q^*}{ N} = q_{B} < r_{B}.
\end{equation*}
By \eqref{interv:q}, we obtain that $q^*>N$ so $\frac{(N-1)q^*}{ N}>N-1$. Thus, we only need to check
\begin{align*}
q^*<\frac{Nr_{B} }{N-1}&\iff \frac{1}{q}-\frac{1}{N}>\frac{N-1}{Nr_{B} }\iff \frac{1}{q}>\frac{N-1}{Nr_{B} }+\frac{1}{N} \\
&\iff  q<\frac{Nr_{B}}{N-1+r_{B}},
\end{align*}
from which, using \eqref{interv:q},  and  that 
\begin{equation}\label{N:rB:N}
\frac{Nr_{B}}{N-1+r_{B}}<N  ,  
\end{equation}
we conclude \eqref{interv:q:2}.\\
   
{\bf Step 2.} {\it Gagliardo–Nirenberg interpolation inequality.}
\bigskip

The Gagliardo-Nirengberg's interpolation inequalities (see \cite{Nirenberg_1959}), implies that there exist a constant $C=C(N,q,|\Om | )$, such that 
\begin{equation}\label{des:G-N}
\| u\|_{L^\infty (\Om ) }\leq C \| u\|_{W^{1,m}(\Om )}^{\sigma }\| u\|^{1-\sigma}_{L^{2^*}(\Om )},
\end{equation}
where 
\begin{equation}\label{def:s}
 \frac{1}{\sigma }=1+2^*\left(\frac{2}{N}-\frac{1}{q} \right).
\end{equation}
From \eqref{equiv:q:qB}, due to the definition of $2^*_{N/q}$, see \eqref{def:N/p&N/pB}, it is easy to check that
\begin{equation}\label{def:s2}
\frac{1}{\sigma }=1+2^*\left(\frac{2}{N}\mp 1-\frac{1}{q} \right)=   2^*_{N/q}-1.
\end{equation}\\

Substituting the estimate of $\|u\|_{W^{1,m}(\Om)}$, see \ref{RE}, and using \eqref{des:1} in the inequality \eqref{des:G-N}, we get 
\begin{align}\label{GN1}
\| u\|_{L^\infty (\Om ) }& \leq C\left[ M^{1-\frac{t}{q}} \| a\|_{L^{r}(\Om )} \left( 1 +\| u\|_{H^{1}(\Om )} ^{2^*\left( \frac{1}{q}-\frac{1}{r}\right)}\right)\right. \nonumber \\
&\quad\left. + M_{B}^{1-\frac{t_B}{q_{B}}} \| a_B\|_{L^{r_B}(\p \Om )} \left( 1+\| u\|^{2_*\left( \frac{1}{q_B}-\frac{1}{r_B}\right) }_{H^{1}(\Om )} \right) \right] ^{\sigma } \| u\|_{L^{2^*}(\Om )}^{(1-\sigma) }\nonumber \\
&\leq  C\left[  M^{\left( 1-\frac{t}{q}\right)\sigma } \| a\|^{\sigma }_{L^{r}(\Om )} \left( 1 +\| u\|_{H^{1}(\Om )} ^{2^*\left( \frac{1}{q}-\frac{1}{r}\right)\sigma }\right) \right. \nonumber \\
&\quad\left.  +  M_{B}^{\left( 1-\frac{t_B}{q_{B}}\right) \sigma } \| a_B\|^{\sigma }_{L^{r_B}(\p \Om )}  \left( 1+\| u\|^{2_*\left( \frac{1}{q_B}-\frac{1}{r_B}\right)\sigma }_{H^{1}(\Om )} \right) \right] \| u\|_{L^{2^*}(\Om )}^{(1-\sigma) },
\end{align}

We now look closely at the exponents of $\| u\|_{L^\infty (\Om ) }$ in the RHS, in order to achieve our estimates.\\

Taking into account the definitions of $M$ and $M_B$, see \eqref{def:M:MB}, that $f$ and $f_B$ are non decreasing, and the definition of the functions $h$ and $h_B$, see \eqref{def:h&hB}, we can write the following relation between them, 
\begin{equation}\label{M:h,MB:hB}
M=\frac{\| u\|_{L^\infty (\Om ) }^{2^*_{N/r}-1}}{h(\| u\|_{L^\infty (\Om ) } )} \qquad \text{and} \qquad M_{B}=\frac{\| u\|_{L^\infty (\Om ) }^{2_{*,N/r_B}-1}}{h_{B}(\| u\|_{L^\infty (\Om ) } )}.
\end{equation}
Moreover, using the definitions of $t$, see \eqref{def:t}, and of $2^*_{N/p}$, see \eqref{def:N/p&N/pB}, we can get  
\begin{equation} \label{Ig:1-t/q}
1-\frac{t}{q}=1-\frac{2^*}{2_{N/r}^*-1}\left( \frac{1}{q}-\frac{1}{r}\right) = \frac{2^*_{N/q}-1}{2^*_{N/r}-1}.
\end{equation}
Thus, because of the expression \eqref{Ig:1-t/q}, we deduce
\begin{equation*}
\left( 2^*_{N/r}-1\right) \left( 1-\frac{t}{q}\right) = (2_{*,N/q}-1) , 
\end{equation*}
and because of the definition of $\s$, see \eqref{def:s2} 

\begin{equation}\label{TRM3}
\left( 2^*_{N/r}-1\right) \left( 1-\frac{t}{q}\right)\sigma =1.
\end{equation}\\

Similarly, from the definitions of $t_B$, see \eqref{def:tB}, and of $2_{*,N/q_B}$, see
\eqref{def:N/p&N/pB}, we obtain
\begin{equation} \label{Ig:1-tB/qB}
1-\frac{t_B}{q_B}=1-\frac{2_*}{2_{*,N/r_B}-1}\left( \frac{1}{q_B}-\frac{1}{r_B}\right)= \frac{2_{*,N/q_B}-1}{2_{*,N/r_B}-1}.   
\end{equation}
Likewise, since \eqref{Ig:1-tB/qB}, the definition of $\s$, see \eqref{def:s2}, and the equivalences  \eqref{equiv:q:qB},  we get
\begin{equation}\label{TRM4}
\left( 2_{*,N/r_B}-1\right) \left( 1-\frac{t_B}{q_B}\right)\sigma =\frac{2_{*,N/q_B}-1}{2_{*,N/q}-1}=1.
\end{equation}

Now, we  divide both sides of the inequality \eqref{GN1} by $\|u\|_{L^\infty (\Om ) }$. Using the definitions of  $M$ and $M_B$,  also the two expressions concerning $\s$; \eqref{TRM3}, \eqref{TRM4}, and the definition of  $a_M$, see \eqref{def:aM}, we obtain that

\begin{equation}\label{1:caM}
1\leq Ca_{M}^{\sigma } \left( \frac{ \left( 1+\|u\|_{H^{1}(\Omega  )}^{2^*\left( \frac{1}{q}-\frac{1}{r} \right) \sigma }\right) }{h^{ \frac{1}{2^*_{N/r}-1}  }(\| u\|_{L^\infty (\Om ) })}+     \frac{\left( 1+\| u\|_{H^{1}(\Omega )}^{2_*\left( \frac{1}{q_B}-\frac{1}{r_B}\right)\sigma }\right)}{h_{B}^{\frac{1}{ 2_{*,N/r_B}-1}}(\| u\|_{L^\infty (\Om ) })}       \right)\| u\|_{L^{2^*}(\Om )}^{(1-\sigma) } .
\end{equation}

Now,  the definition of $h_m$ (see \eqref{def:hm}), implies that
\begin{equation*}
\frac{1}{h_{m}^{\frac{1}{2^*_{N/r}-1}}(\| u\|_{L^\infty (\Om ) } )}=\max \left\{ \frac{1}{h^{ \frac{1}{2^*_{N/r}-1}  }(\| u\|_{L^\infty (\Om ) })}, \frac{1}{h_{B}^{\frac{1}{ 2_{*,N/r_B}-1}}(\| u\|_{L^\infty (\Om ) })} \right\} .
\end{equation*}
So, substituting this maximum in the inequality \eqref{1:caM}, we get
\begin{align}\label{des:hm<caM(exp)}
h_{m}^{\frac{1}{2^*_{N/r}-1} }(\| u\|_{L^\infty (\Om ) })&\leq Ca_{M}^{\sigma } \left(  1+\| u\|_{H^{1}(\Om )}^{2^*\left( \frac{1}{q}-\frac{1}{r} \right) \sigma } \right. \\ 
&\qquad \left. +\| u\|_{H^{1}(\Om )}^{2_*\left( \frac{1}{q_B}-\frac{1}{r_B}\right)\sigma }\right) \| u\|_{L^{2^*}(\Om )}^{(1-\sigma) }.  \nonumber
\end{align}

The RHS in the above inequality is upper bounded by a term with the largest exponent of both addends. Let us denote this maximum by
\begin{equation}\label{def:EM}
E_M:=\max \left\{2^*\left( \frac{1}{q}-\frac{1}{r} \right) 
, 2_*\left( \frac{1}{q_B}-\frac{1}{r_B}\right) \right\}.
\end{equation}
Since the inequality \eqref{des:hm<caM(exp)} the definition \eqref{def:EM} and Sobolev's embedding, we obtain
\begin{equation}\label{estim}
h_{m}(\| u\|_{L^\infty (\Om ) })\leq C a_{M}^{\theta }  \left( 1+\| u\|_{H^{1}(\Om )}^{\beta }\right),
\end{equation}
where
\begin{equation}\label{def:theta}
\theta :=\left(2^*_{N/r}-1\right)\s=\frac{2^*_{N/r}-1}{2^*_{N/q}-1},
\end{equation}
and
\begin{equation}\label{def:beta}
\be :=\left(E_M+\frac{1-\s}{\sigma}\right)\te .
\end{equation}

At this moment, we look closely at the definition of $E_M$.

\noindent Firstly by definition of $2^*$ and of $2_*$, see \eqref{def:N/p&N/pB}, secondly by election of $q_B$, see \eqref{def:qB}, and finally rearranging terms, we observe that
\begin{align} \label{des:Exp1>Exp2}
& 2^*\left( \frac{1}{q}-\frac{1}{r} \right) 
\ge 2_*\left( \frac{1}{q_B}-\frac{1}{r_B}\right)  \\
& \iff  \frac{1}{q} \mp \frac{1}{N}-\frac{1}{r}
\ge \frac{N-1}{N}\left( \frac{1}{q_B}-\frac{1}{r_B}\right)  \nonumber \\
&  \iff 
\frac1{r} -\frac1{N}\le \frac{N-1}{Nr_B}.  \nonumber
\end{align}
If $r\ge N$,  then $1/r- 1/N\le 0$, and the last inequality holds. Moreover,   if  
$N/2<r< N$, then the last inequality holds if and only if 
\begin{equation*}
\left(1/r- 1/N\right)^{-1}=:r^*\ge Nr_B/(N-1).
\end{equation*}
Observe that
\begin{equation*}
r^*\ge Nr_B/(N-1)\iff 2^*_{N/r} \ge 2_{*,N/r_B}.
\end{equation*}
On the contrary, the reverse inequality to \eqref{des:Exp1>Exp2}, will be satisfied whenever $N/2<r< N$, and $r^*\le Nr_B/(N-1)$. Hence
\begin{equation}\label{def:EM:2}
E_M =
\begin{cases}
2^*\left( \frac{1}{q}-\frac{1}{r} \right)  & \begin{cases}
\text{if }r\ge N,\\ 
\text{or } N/2<r< N \text{ and }r^*\ge \frac{Nr_B}{N-1},
\end{cases}\\
2_*\left( \frac{1}{q_B}-\frac{1}{r_B}\right)\ &\ \text{   if } N/2<r< N \text{ and }r^*\le \frac{Nr_B}{N-1}.
\end{cases}   
\end{equation}\\

Consequently, we have two cases in the  search for the optimum exponents $\te$ and $\be$ varying $q$, see \eqref{interv:q:2}:\\

\begin{enumerate}
\item[{\bf Case (I):}] {\it Either $r\ge N,$ or $N/2<r< N $ and $r^*\ge \frac{Nr_B}{N-1}$.}\\
\end{enumerate}   
Using the definition of $\be$, \eqref{def:beta}, the first equality for $E_M$ in \eqref{def:EM:2}, and the expression for $\s $, see\eqref{def:s}, and for $2^{*}_{N/r}$, see \eqref{def:N/p&N/pB},
\begin{align}\label{Ig:Beta1}
\be=&\left[2^*\left( \frac{1}{q}-\frac{1}{r} \right) +\frac{1-\s}{\sigma}\right]\te \\
\label{Ig:Beta1b}
=&\left[2^*\left( \frac{1}{q}-\frac{1}{r} \right) +2^*\left(\frac{2}{N}-\frac{1}{q} \right)\right]\te =(2^*_{N/r}-2)\theta .
\end{align}
The function $\theta  :q\mapsto \theta (q) $, defined by \eqref{def:theta} is decreasing. We look for their infimum for $q$ in the interval \eqref{interv:q}. 

{\it Assume $r\ge N$. }
Since \eqref{interv:q:2}--\eqref{N:rB:N}, we deduce that $q\in \left(\frac{N}{2},\frac{Nr_{B}}{N-1+r_{B}}\right) $.

{\it Assume  $N/2<r< N $ and $r^*\ge \frac{Nr_B}{N-1}$.}
Note that 
\begin{align}
\label{r*:rB}
r^*\ge \frac{Nr_B}{N-1} \iff \frac{1}{r}-\frac{1}{N}\le \frac{N-1}{Nr_B}\iff r\ge \frac{Nr_{B}}{N-1+r_{B}},    
\end{align}
and also $q\in \left(\frac{N}{2},\frac{Nr_{B}}{N-1+r_{B}}\right) $.

Hence, in case I,
\begin{equation} \label{te:A}
\inf _{q\in \left(\frac{N}{2},\frac{Nr_{B}}{N-1+r_{B}}\right) }\theta (q)=\theta \left(\frac{Nr_{B}}{N-1+r_{B}}\right)=\frac{\frac{1}{2}-\frac{N-r}{Nr}}{\frac{1}{2}-\frac{N-1}{Nr_{B}}}.
\end{equation}

\begin{enumerate}
\item[{\bf Case (II):}] {\it $N/2<r< N $ and $r^*\le \frac{Nr_B}{N-1}$.}\\
\end{enumerate} 
Likewise, using the definition of $\be$, \eqref{def:beta}, the second equality for $E_M$ in \eqref{def:EM:2}, and  the expressions for $\s$ \eqref{def:s}, for $2^*_{N/p},\ 2_{*,N/ p_{B}}$ \eqref{def:N/p&N/pB}, and the equivalence \eqref{equiv:q:qB}
\begin{align}\label{Ig:Beta2}
\beta  
&=\left[2_*\left( \frac{1}{q_B}-\frac{1}{r_B} \right) +\frac{1-\s}{\sigma}\right]\te \nonumber \\
&=\left[2_*\left( \frac{1}{q_B}\mp 1-\frac{1}{r_B} \right)  +2^*\left(\frac{2}{N} \mp 1-\frac{1}{q} \right)\right]\te \nonumber\\
&=\left[ 2_{*,N/r_B }-2_{*,N/q_B }-2+2^*_{N/q}\right] \te=\left[ 2_{*,N/r_B }-2\right] \te.
\end{align} 
Now, for $q$ satisfying \eqref{interv:q:2}, thanks to \eqref{r*:rB}, we deduce that $q\in (\frac{N}{2},r ) $, hence
\begin{equation}\label{te:1}
\inf _{q\in \left(\frac{N}{2},r\right)}\theta (q)=\theta (r)=1.
\end{equation}\\

Finally, we introduce into the inequality \eqref{estim}, the infima of $\te$ and $\be$ given by \eqref{te:A} and \eqref{Ig:Beta1b} respectively in case I, and by \eqref{te:1} and \eqref{Ig:Beta2}, in case II. Since  these infima are not attained in the set where $q$ belongs, for any $\e >0$, there exists a constant $C_{\e } > 0$ such that,
\begin{equation*}
h_{m}(\| u\|_{L^\infty (\Om ) })\leq C_{\e } a_{M}^{A+\e }\left( 1+\| u\|_{H^{1}(\Om )}^{(2^*_{N/r}-2)(A+\e )}\right) ,
\end{equation*}
where $A$ is defined in \eqref{def:A}, and $C_{\e }=C_{\e } (\e , N, |\Om |,|\p \Om |  )$ and it is independent of $u$.\\
\end{proof}

In the following Remark, we observe the necessity of the election for $q_B$, see \eqref{def:qB}.

\begin{rem} \label{rem1}
Assume for a while that \eqref{def:qB} do not hold and, to fix ideas, that     
\begin{align*}
q_{B}&< \frac{(N-1)q^*}{N}\implies m=\frac{Nq_{B}}{N-1} >N.
\end{align*}
We also have the following equivalence
\begin{equation*}
q_{B}<\frac{(N-1)q^*}{N}
\iff \frac{2^*}{q^*}<\frac{2_*}{q_{B}}
\iff 2_{*,N/q_B}<  2^*_{N/q}. 
\end{equation*}
Indeed, the first equivalence is obvious. With respect to the second one, notice that, due to the definitions of $2^*_{N/q}$ and of $2_{*,N/q_B}$, see \eqref{def:N/p&N/pB}, 
we can conclude that
\begin{align*}
2_{*,N/q_B} &= 2_*-\frac{2_*}{q_B} < 2_*-\frac{2^*}{q^*} = 2^*-\frac{2^*}{q} = 2^*_{N/q}.
\end{align*}
Now, in the Gagliardo-Nirengberg's interpolation inequality,  see \eqref{des:G-N}, 
the parameter $\s$ is given by
\begin{align*}
\frac{1}{\sigma }&=1+2^*\left(\frac{1}{N}-\frac{1}{m} \right)=1+2^*\left(\frac{1}{N}\mp 1-\frac{N-1}{Nq_B} \right)\\
&=   2_{*,N/q_B}-1< 2^*_{N/q}-1 .
\end{align*}

And the expression \eqref{TRM3} becomes now
\begin{equation*}
\left( 2^*_{N/r}-1\right) \left( 1-\frac{t}{q}\right)\sigma =\frac{2_{*,N/q}-1}{2_{*,N/q_B}-1}>1.
\end{equation*}

The above inequality means that the exponent of $\| u\|_{L^\infty (\Om ) }$ in the RHS will dominate 1,  which is the exponent of $\| u\|_{L^\infty (\Om ) }$ in the LHS, and the bounds can not be reached.
\smallskip

Likewise, if $q_{B}> \frac{(N-1)q^*}{N}$, then $m=q^*$ and it can be proved that
\begin{equation*}
\frac{1}{\sigma }=1+2^*\left(\frac{2}{N}\mp 1-\frac{1}{q} \right)=   2^*_{N/q}-1<2_{*,N/q_B}-1,
\end{equation*}
and so
\begin{equation*}
\left( 2_{*,N/r_B}-1\right) \left( 1-\frac{t_B}{q_B}\right)\sigma = (2_{*,N/q_B}-1)\s=\frac{2_{*,N/q_B}-1}{2^*_{N/q}-1} > 1,
\end{equation*}
concluding that necessarily,  $q_B$ have to be chosen as in  \eqref{def:qB}.
\end{rem}

Throughout that proof, we have explicit  estimates of $h_{m}(\| u\|_{L^\infty (\Om ) }) $ expressed in  their $L^{2^*}(\Om)$ norm and $L^{2_*}(\p\Om)$ norm (see \eqref{S3-2} and \eqref{des:normftil}).
Previously, we unify those estimates in  their $H^{1}(\Om)$ norm to simplify the expression.\\
In the next Corollary, we split those  estimates in terms of the $L^{2^*}(\Om)$ norm and the $L^{2_*}(\p\Om)$ norm.

\begin{cor}
Assume that all the hypotheses of Theorem \ref{th:CotasLinf} hold. Then, for all $\, \e >0$, there exist $C_{\e }>0$ depending of $\e $,  $N$, $|\Om | $ and $|\p \Om | $, but independent of $u$, such that
\begin{align*}
h_{m}(\| u\|_{L^\infty (\Om ) })&\leq Ca_{M}^{A+\e} \left(  1+\| u\|_{L^{2^*}(\Om )}^{A_1+\e} +\| u\|_{L^{2_*}(\p\Om )}^{A_2+\e}\| u\|_{L^{2^*}(\Om )}^{A_3 +\e}\right) ,  \nonumber
\end{align*}
where  $A$  is defined in \eqref{def:A},
\begin{equation*}
\begin{aligned}
&\begin{array}{ll}
A_1:=&(2^*_{N/r}-2)A\\ 
A_2:=&0 \\
A_3:=&(2_{*,N/r_B}-2)A
\end{array} 
& \text{if } \begin{cases}
\text{either }   r\ge N,\\[.2cm]
\text{or } N/2<r< N  \text{ and } r^*\ge \frac{Nr_B}{N-1},
\end{cases}\\[1em]
&\begin{array}{ll}
A_1:=&2^*_{N/r}-2  \\
A_2:=&2_{*,N/r_B}-2^*_{N/r} \\
A_3:=&2^*_{N/r}-2
\end{array} 
& \mspace{-270mu} \text{if } \qquad N/2<r< N \text{ and } r^*\le \frac{Nr_B}{N-1}.\\
\end{aligned}
\end{equation*}
\end{cor}

\begin{proof}
The proof is similar to the proof of the Theorem \eqref{th:CotasLinf}.\\

{\bf Step 1.} {\it $W^{1,m}(\Om )$ estimates for $m>N$.}\\

Substituting \eqref{S3-2} in the second factor on the RHS of \eqref{ts'}, and this is \eqref{S3-1}, we get 
\begin{equation}\label{Lq:norm2}
\left(\int_{\Om} |f(\cdot ,u)|^{q}dx \right) ^{\frac{1}{q}}\leq C M^{1-\frac{t}{q}} \| a\|_{L^{r}(\Om )} \left( 1 +\| u\|_{L^{2^*}(\Om )} ^{2^*\left( \frac{1}{q}-\frac{1}{r}\right)} \right) , 
\end{equation}
with $M$  defined  in \eqref{def:M:MB},  $t$  in \eqref{def:t} and  $q$ in \eqref{interv:q}, respectively. See  the analogy with \eqref{Lq:norm}. \\

On the other hand, replacing \eqref{des:normftil} in the second factor on the RHS of \eqref{tBsB'}, and this in \eqref{L:qB}, we get
\begin{align}\label{LqB:norm2}
\left(\int_{\p \Om }|\tilde{f}_{B}(\cdot ,u)|^{q_{B}}dx \right) ^{\frac{1}{q_{B}}}&\leq C M_{B}^{1-\frac{t_B}{q_{B}}} \| a_B\|_{L^{r_B}(\p \Om )}\\
&\qquad \times\left( 1+\| u\|^{2_*\left( \frac{1}{q_B}-\frac{1}{r_B}\right) }_{L^{2_*}(\p \Om )} \right) , \nonumber
\end{align}
with $M_{B}$ defined  in \eqref{def:M:MB}, $t_{B}$  in \eqref{def:tB},  and
$q_{B}$  in \eqref{interv:qB} respectively.
See the analogy with\eqref{LqB:norm}.
\medskip

By elliptic regularity, we estimate the norm $\| u\|_{W^{1,m}(\Om)}$ in terms of  \eqref{Lq:norm2} and \eqref{LqB:norm2}, see Theorem \ref{Th:RegLineal}, obtaining

\begin{align}\label{RE2}
\| u\|_{W^{1,m}(\Om )}&\leq C \left[ M^{1-\frac{t}{q}} \| a\|_{L^{r}(\Om )} \left( 1 +\| u\|_{L^{2^*}(\Om )} ^{2^*\left( \frac{1}{q}-\frac{1}{r}\right)}\right) \right. \nonumber\\
&  \left. \quad \quad +  M_{B}^{1-\frac{t_B}{q_{B}}} \| a_B\|_{L^{r_B}(\p \Om )}\left( 1+\| u\|^{2_*\left( \frac{1}{q_B}-\frac{1}{r_B}\right) }_{L^{2_*}(\p \Om )} \right) \right], 
\end{align}
with $m>N.$ See also the analogy with \eqref{RE}.\\

{\bf Step 2.} {\it Gagliardo–Nirenberg interpolation inequality.}\\

Substituting \ref{RE2} in the Gagliardo-Nirengberg's  inequality \eqref{des:G-N} and using the inequality \eqref{des:1}, we get 

\begin{align}\label{GN1:2}
\| u\|_{L^\infty (\Om ) }& \leq  C\left[  M^{\left( 1-\frac{t}{q}\right)\sigma } \| a\|^{\sigma }_{L^{r}(\Om )} \left( 1 +\| u\|_{L^{2^*}(\Om )} ^{2^*\left( \frac{1}{q}-\frac{1}{r}\right)\sigma }\right) \right. \nonumber \\
&\quad \left.  +  M_{B}^{\left( 1-\frac{t_B}{q_{B}}\right) \sigma } \| a_B\|^{\sigma }_{L^{r_B}(\p \Om )}  \left( 1+\| u\|^{2_*\left( \frac{1}{q_B}-\frac{1}{r_B}\right)\sigma }_{L^{2_*}(\p \Om )} \right) \right] \| u\|_{L^{2^*}(\Om )}^{(1-\sigma) },
\end{align}

Using the definitions of  $M$ and $M_B$, see \eqref{M:h,MB:hB}, using also \eqref{TRM3}, \eqref{TRM4}, the definition of  $a_M$ (see \eqref{def:aM}), and dividing both sides of the inequality \eqref{GN1:2} by $\|u\|_{L^\infty (\Om ) }$, we obtain that

\begin{equation*}
1\leq Ca_{M}^{\sigma } \left( \frac{ \left( 1+\|u\|_{L^{2^*}(\Omega  )}^{2^*\left( \frac{1}{q}-\frac{1}{r} \right) \sigma }\right) }{h^{ \frac{1}{2^*_{N/r}-1}  }(\| u\|_{L^\infty (\Om ) })}+     \frac{\left( 1+\| u\|_{L^{2_*}(\p \Omega )}^{2_*\left( \frac{1}{q_B}-\frac{1}{r_B}\right)\sigma }\right)}{h_{B}^{\frac{1}{ 2_{*,N/r_B}-1}}(\| u\|_{L^\infty (\Om ) })}       \right)\| u\|_{L^{2^*(\Om) }}^{(1-\sigma)} .
\end{equation*}
Then,
\begin{align*}
h_{m}^{\frac{1}{2^*_{N/r}-1} }(\| u\|_{L^\infty (\Om ) })&\leq Ca_{M}^{\sigma } \left(  2+\| u\|_{L^{2^*}(\Om )}^{2^*\left( \frac{1}{q}-\frac{1}{r} \right) \sigma } \right. \\ 
&\quad \left. +\| u\|_{L^{2_*}(\p\Om )}^{2_*\left( \frac{1}{q_B}-\frac{1}{r_B}\right)\sigma }\right) \| u\|_{L^{2^*}(\Om )}^{(1-\sigma) }.  \nonumber
\end{align*}
where $h_m$ is defined in \eqref{def:hm}. See  the analogy with \eqref{des:hm<caM(exp)}.\\
Clearing $h_m$, we get 
\begin{align*}
h_{m}(\| u\|_{L^\infty (\Om ) })&\leq Ca_{M}^{\sigma(2^*_{N/r}-1) } \left(  1+\| u\|_{L^{2^*}(\Om )}^{2^*\left( \frac{1}{q}-\frac{1}{r} \right)(2^*_{N/r}-1) \sigma} \right. \\ 
&\quad \left. +\| u\|_{L^{2_*}(\p\Om )}^{2_*\left( \frac{1}{q_B}-\frac{1}{r_B}\right) (2^*_{N/r}-1)\sigma }\right) \| u\|_{L^{2^*}(\Om )}^{(1-\sigma)(2^*_{N/r}-1)} . \nonumber
\end{align*}
Substituting in the exponents the parameter $\theta$ (see \eqref{def:theta}), we get

\begin{align}\label{hm1:cor}
h_{m}(\| u\|_{L^\infty (\Om ) })&\leq Ca_{M}^{\theta } \left(  1+\| u\|_{L^{2^*}(\Om )}^{\left[2^*\left( \frac{1}{q}-\frac{1}{r} \right) + \frac{1-\sigma}{\sigma }\right] \theta} \right. \nonumber \\ 
&\quad \left. +\| u\|_{L^{2_*}(\p\Om )}^{2_*\left( \frac{1}{q_B}-\frac{1}{r_B}\right) \theta }\| u\|_{L^{2^*}(\Om )}^{\frac{1-\sigma}{\sigma }\theta }\right).  
\end{align}

Let us define the function $\theta _{1}  =\theta _{1}(q)$ as the first exponent inside the brackets. Using the definitions of $2^*_{N/q}$, see \eqref{def:N/p&N/pB}, and of $\s$, see \eqref{def:s2}, we get
\begin{align*}
\te_1(q):&=\left[2^*\left( \frac{1}{q}-\frac{1}{r} \right) + \frac{1-\sigma}{\sigma }\right] \theta(q) \nonumber\\
&=(2^*_{N/r}-2) \theta (q),
\end{align*}
note that this value is equal to $\be $ in case I of Theorem \eqref{th:CotasLinf}, see \eqref{Ig:Beta1}--\eqref{Ig:Beta1b}. \\

We define the function $\theta _{2}  =\theta _{2}(q) $ as the second exponent. By the definition of $2_{*,N/ q_{B}}$, see \eqref{def:N/p&N/pB}, and the equivalence \eqref{equiv:q:qB}
\begin{align*}
\te_2(q):&=2_*\left( \frac{1}{q_B}\mp1-\frac{1} {r_B}\right) \theta(q)\\
&=(2_{*,N/r_B}-2^*_{N/q})\theta(q).\nonumber
\end{align*} 
We define the function $\theta _{3}=\theta _{3}(q) $ as the third exponent. Using the expression \eqref{def:s2} for $\s$,   
\begin{equation*}
\te _3(q) :=\left(\frac{1-\sigma}{\sigma }\right)\theta (q)  =(2^*_{N/q}-2)\te (q).
\end{equation*}

As before, let $q_B,$ $\theta $ be defined by \eqref{def:qB}, and \eqref{def:theta} respectively. The function 
\begin{equation*}
    (\te_1+\te_2+\te_3)(q)= (2^*_{N/r}+2_{*,N/rB}-4)\te (q),
\end{equation*}
is decreasing, and we look for their infimum  for $q$ in the interval  \eqref{interv:q:2}.  
Thus, as before, we consider the previous two cases.\\
\begin{enumerate}
\item[{\bf Case (I)}] {\it Either $r\ge N,$ or $N/2<r< N $ and $r^*\ge \frac{Nr_B}{N-1}$.}
\end{enumerate}  

In this case, $q\in \left(\frac{N}{2},\frac{Nr_{B}}{N-1+r_{B}}\right).$ For $A$ is defined in \eqref{def:A},
the  exponents are given by 
\begin{align*}
A'_{1}&:=\te_1 \left(\frac{Nr_{B}}{N-1+r_{B}}\right)=(2^*_{N/r}-2) A,
\end{align*}
by
\begin{align*}
A'_{2}&:=\te_2 \left(\frac{Nr_{B}}{N-1+r_{B}}\right)=\frac{2}{N-2}\left( \frac{N-1+r_{B}}{r_{B}}-1-\frac{N-1}{r_B}\right)A =0,
\end{align*}
and by
\begin{align*}
A'_{3}:=\te _3 \left(\frac{Nr_B}{N-1+r_B}\right)&= \left(2^*\left(1-\frac{N-1+r_B}{Nr_B}\right)-2\right)A\\
&=\left(\frac{2(N-2)}{N-2}\left(\frac{r_B-1}{r_B}\right)-2\right)A\\
&=(2_{*,N/r_B}-2)A
\end{align*}
Hence, the inequality \eqref{hm1:cor} can be rewritten as 
    \begin{align*}
h_{m}(\| u\|_{L^\infty (\Om ) })&\leq Ca_{M}^{A+\e } \left(  1+\| u\|_{L^{2^*}(\Om )}^{(2^*_{N/r}-2) A + \e }\right. \\
&\left. \qquad +\| u\|_{L^{2_*}(\p\Om )}^{\e}\| u\|_{L^{2^*}(\Om )}^{(2_{*,N/r_B}-2)A+\e }\right) ,  \nonumber
\end{align*}
where $A$ is defined in \eqref{def:A}.\\

\begin{enumerate}
\item[{\bf Case (II)}] {\it $N/2<r< N $ and $r^*\le \frac{Nr_B}{N-1}$.}
\end{enumerate} 

In that case, $q\in \left(\frac{N}{2},r\right) $ (see \eqref{interv:q:2}). 
The exponents are given by
\begin{align*}
A''_{1}&:=\te_1 \left(r\right)=2^*_{N/r}-2,\\
A''_2&:=\te _2(r)
=\frac{2}{N-2}\left( \frac{N}{r}-1 \mp N-\frac{N-1} {r_B}\right)
= 2_{*,N/r_B}-2^*_{N/r},\\
A''_3&:=\te_3(r)=2^*_{N/r}-2.
\end{align*}
Therefore, the inequality \eqref{hm1:cor} is rewritten as
    \begin{align*}
h_{m}(\| u\|_{L^\infty (\Om ) })&\leq C_{\e } a_{M}^{1+\e }\left( 1+\| u\|_{L^{2^*}(\Om )}^{2^*_{N/r}-2+ \e } \right. \\ 
&\left. \qquad +\|u\|_{L^{2_*}(\p \Om)}^{2_{*,N/r_B}-2^*_{N/r}+\e }\| u\|_{L^{2^*}(\Om )}^{2^*_{N/r}-2+\e }\right) .
\end{align*}

\end{proof}

The next corollary proves that any sequence $\lbrace u_{k}\rbrace \subset  H^{1}(\Om )$ of weak solution to \eqref{E0.1}, uniformly bounded in the  $L^{2^*}(\Om)$ norm and in the $L^{2_*}(\p\Om)$ norm, is also uniformly bounded in the $C(\Omb)$-norm.

\begin{cor}
\label{cor:2}
Let $f:\Om \times \mathbb{R}\rightarrow \mathbb{R}$ and $f_{B}:\p \Om \times \mathbb{R}\rightarrow \mathbb{R}$ be Carath\'eodory functions, satisfying {\rm \ref{f1}--\ref{f2}} and {\rm \ref{fB1}--\ref{fB2}}, respectively. Let $\lbrace u_{k}\rbrace \subset H^{\, 1}(\Omega )$ be a sequence of weak solutions to \eqref{E0.1} satisfying that, there exist $C_{0}>0$, such that 
\begin{equation*}
\|u_{k}\|_{L^{2^*}(\Om)}\leq C_{0} \quad \text{and} \quad \|u_{k}\|_{L^{2_*}(\p \Om)}\leq C_{0}.
\end{equation*}
Then, there exist $C>0$ such that,
\begin{equation*}
\|u_{k}\|_{C(\Omb)}\le C.
\end{equation*}
\end{cor}

\begin{proof}
We proceed by contradiction, assuming that $\|u_{k}\|_{L^\infty (\Om)}\rightarrow \infty $.
By the Theorem \eqref{th:CotasLinf}  and the remark \ref{rem1}, we get 
\begin{equation}\label{des:hm:C}
h_m(\|u_{k}\|_{C(\Omb)})\le C, \quad \text{for} \quad C>0,
\end{equation}
where, $h_{m}$ is defined in \eqref{def:hm}.
    
Using  \eqref{h:infty}, we deduce that  $h_m(\|u_{k}\|_{C(\Omb)})\rightarrow \infty$ as $k\rightarrow \infty $, which contradicts \eqref{des:hm:C}.\\

\end{proof}

\begin{cor}
Let $f:\Om \times \mathbb{R}\rightarrow \mathbb{R}$ and $f_{B}:\p \Om \times \mathbb{R}\rightarrow \mathbb{R}$ be Carath\'eodory functions, satisfying {\rm \ref{f1}--\ref{f2}} and {\rm \ref{fB1}--\ref{fB2}} respectively. Let $\lbrace u_{k}\rbrace \subset H^{\, 1}(\Omega )$ be a sequence of weak solutions to \eqref{E0.1}. Then, the following statements are equivalent 

\begin{enumerate}
\item[i):] $\|u_{k}\|_{L^{2^*}(\Om)}\leq C_{1} \quad \text{and} \quad \|u_{k}\|_{L^{2_*}(\p \Om)}\leq C_{1}$.
\item[ii):] $\|u_{k}\|_{C(\Omb)}\leq C_{3}$.
\item[iii):] $\|u_{k}\|_{H^{1}(\Om)}\leq C_{2}$.
\end{enumerate}
for some constants $C_i$  independent of $k$, $i=1,2 ,3.$
\end{cor}

\begin{proof}
We prove that $i) \Rightarrow ii) \Rightarrow iii)  \Rightarrow i)$.

The proof of $i) \Rightarrow ii) $ follows directly from the Corollary \ref{cor:2}.

Now, using the elliptic regularity result, see the estimate \eqref{estim:W1m} in the Theorem \ref{th:reg}, and the Gagliardo-Nirenberg interpolation, the proof of $ii) \Rightarrow iii) $ is done.

Finally,  Sobolev's embedding and the continuity of the trace operator, proves that $iii)  \Rightarrow i)$.

\end{proof}

\begin{appendix}

\section{Regularity for the Neumann non homogeneous linear problem}
\label{ApB}
\renewcommand{\thesection}{\Alph{section}}
\counterwithin{equation}{section}
\setcounter{equation}{0}
In this appendix,  we recall the regularity of weak solution to the linear problem with non homogeneous data both at the interior and on the boundary. 

Let us consider the linear nonhomogeneous Neumann problem 
\begin{equation}\label{ApB1}
\left \{
\begin{aligned}
-\Delta u +u=& g(x), \; \; \; x\in \Om , \\
\frac{\p u}{\p \eta  }=& g_{B}(x), \;x\in \p \Om ,
\end{aligned}
\right . 
\end{equation}

where $\Om \subset \mathbb{R}^{N}$, ($N>2$), is an open, connected and bounded domain with $C^{2}$ boundary.

\begin{thm}\label{Th:RegLineal} 
Let us consider the problem \eqref{ApB1}, there exist a positive constant $C>0$ independent of $u$, $h$ and $g_{B}$ such that the following holds:

\begin{enumerate}[label=\rm{(\roman*)}]
\item If $\p \Om \in C^{0,1}$, $g\in L^{q}(\Om )$ and $g_{B}\in L^{q_{B}}(\p \Om )$ with $q\geq 1$ and $q_{B}\geq 1$, then there exist a unique $u\in W^{1,m}(\Om )$ and 
\begin{equation} \label{W:1:m}
\|u\|_{W^{1,m}(\Om )}\leq C\left( \|g\|_{L^{q}(\Om )}+\|g_{B}\|_{L^{q_B}(\p \Om )}\right),
\end{equation}
where $m=\min \lbrace \frac{Nq}{N-q},\frac{Nq_{B}}{N-1}\rbrace $ whenever $1\leq q<N$, or $m=\min \lbrace q,\frac{Nq_{B}}{N-1}\rbrace$ whenever $q\geq N.$ Furthermore, if $q>\frac{N}{2}$ and $q_{B}>N-1$, then
\begin{equation*}
\|u\|_{C^{\nu }(\Omb )}\leq C\left( \|g\|_{L^{q}(\Om )}+\|g_{B}\|_{L^{q_B}(\p \Om )}\right),    
\end{equation*}
where $\nu  =1-\frac{N}{m}$, ($m>N$).

\item  If $\p \Om \in C^{1,1}$, $g\in C^{\nu }(\Omega )\cap L^{q}(\Om )$ and $g_{B}\in L^{q_{B}}(\p \Om )$ with $q>\frac{N}{2}$ and $q_{B}>N-1$, then there exist a unique $u\in C^{\nu }(\Omb)\cap C^{2,\nu  }(\Omega )$.

\item If $\p \Om \in C^{2,\nu  }$, $g\in C^{\nu }(\Omb)$ and $g_{B}\in C^{1,\nu  }(\p \Omega )$ with $\nu  \in (0,1)$, then there exist a unique $u\in C^{2,\nu  }(\Omb)$ and 
\begin{equation*}
\|u\|_{C^{2,\nu  }(\Omb )}\leq C\left( \|g\|_{C^{\nu }(\Omb)}+\|g_{B}\|_{C^{1,\nu  }(\p \Om )}\right),
\end{equation*}
where $C$ is a positive constant independent of $u$, $g$ and $g_{B}$.

\item    If $\p \Om \in C^{2}$, $g\in L^{p}(\Om )$ and $g_{B}\in W^{1-\frac{1}{p},p}(\p \Om )$, then $u\in W^{2,p}(\Om )$ and 
\begin{equation*}
\|u\|_{W^{2,p}(\Om )}\leq C\left( \|g\|_{L^{p}(\Om )}+\|g_{B}\|_{W^{1-\frac{1}{p},p}(\p \Om )}\right),
\end{equation*}
where $C$ is a positive constant independent of $u$, $g$ and $g_{B}$.

\item If $\p \Om \in C^{1,\nu  }$ with $\nu  \in (0,1]$, $g\in C^{\nu }(\Om )$ and $g_{B}\in C^{\nu }(\p \Om )\cap L^{\infty}(\p \Om )$ then if $u$ is a bounded weak solution to \eqref{ApB1}, then $u\in C^{1,\beta }(\Omb)\cap C^{2,\beta }(\Omega )$, where $\beta $ depends on $\nu  $ and $N$. 
\end{enumerate}   
\end{thm}

\begin{proof}
\begin{enumerate}
\item[(i):] It follows from \cite[Ch.3 Sec. 6]{Ladyzhenskaya_Ural’tseva} or \cite[Lem. 2.2]{Mavinga_Pardo} that there exists a unique $u \in W^{1,p}(\Omega)$ solving \eqref{ApB1}. Now if $p > N$, using the Sobolev embedding theorem, one has $u \in C^{\alpha}(\overline{\Omega})$. Then by applying \cite[Thm. 6.13]{Gilbarg_Trudinger} for the corresponding nonhomogeneous Dirichlet problem, we have that $u \in C^{1,\alpha}(\overline{\Omega})$, see also \cite{Mavinga_Pardo}.

\item[(ii):] From part (i) we have that $u \in C^{\alpha}(\overline{\Omega})$. Since $\partial \Omega \in C^{1,1}$, $\Omega$ satisfies the exterior sphere condition at every point on the boundary and using the fact that $g\in C^{\alpha}(\overline{\Omega})$, reasoning as above it follows from \cite[Thm. 6.13]{Gilbarg_Trudinger} that $u \in C^{\alpha}(\overline{\Omega}) \cap C^{2,\alpha}(\overline{\Omega})$.

\item[(iii):] See \cite[Page 55]{Amann1976} or \cite[Chap.3 Sec. 3]{Ladyzhenskaya_Ural’tseva}.

\item[(iv):] See \cite[Page 55]{Amann1976} or \cite[Chap.3 Sec. 9]{Ladyzhenskaya_Ural’tseva}.

\item[(v):] By \cite[Thm. 2]{Lieberman1988}, one has $u \in C^{1,\beta}(\overline{\Omega})$. Then using the bootstrap for the differential equation in $\Omega$, we get the desired regularity in $\Omega$.     
\end{enumerate}
\end{proof}

\section{Regularity of weak solutions}
\label{sec:prelim}
\renewcommand{\thesection}{\Alph{section}}
\numberwithin{equation}{section}
\setcounter{equation}{0}

In this section, we establish auxiliary results on further regularity of weak solutions to \eqref{E0.1}, by assuming that conditions on the growth of the nonlinearities are subcritical or even critical. Using a Moser type procedure, it is known that $u\in L^{q}(\Om )\cap L^{q}(\p \Om )$ for all $q<\infty $ (see \cite[Theorem 3.1]{Marino_Winkert_NonlAnal_2019}). Moreover, using elliptic regularity theory, we state the following result that guarantees, in particular, Hölder regularity of any weak solution to \eqref{E0.1}. 

\begin{thm}\label{th:reg}
Let $\Om \subset \mathbb{R}^{N}$, $f:\Om \times \mathbb{R}\rightarrow \mathbb{R}$ and $f_{B}:\p \Om  \times \mathbb{R}\rightarrow \mathbb{R}$ be Carath\'eodory functions, such that
\begin{align*}
|f(x,s)| &\leq |a(x)| \left( 1+|s|^{2^*_{N/r}-1}\right) \qquad \text{and} \\
|f_{B}(x,s)| &\leq |a_{B}(x)|\left( 1+|s|^{2_{*,N/r_{B}}-1}\right) ,\nonumber
\end{align*}
where 
\begin{align*}
a(x)\in L^{r}(\Om ), \qquad &\text{with} \qquad \frac{N}{2}<r\leq \infty,  \qquad \text{and }\\
a_{B}(x)\in L^{r_B}(\p \Om ), \qquad  &\text{with} \qquad N-1<r_{B}\leq \infty .
\end{align*}
Let $u\in H^{1}(\Om )$ be a weak solution to \eqref{E0.1}, then $u\in L^{q}(\Om )\cap L^{q}(\p \Om )$ for all $1\le q<\infty .$  \\
Moreover,   $u\in W^{1,m}(\Om )\cap  C^{\nu }(\Omb)$, and the following estimates holds
\begin{equation}\label{estim:W1m}
\| u\|_{W^{1,m}(\Om )}\leq C \left( \| f(\cdot,u)\|_{L^{r}(\Om )}+\| f_B(\cdot,u)\|_{L^{r_B}(\p \Om )}\right)
\end{equation}
\qquad and
\begin{equation}\label{estim:C:nu}
\| u\|_{C^{\nu }(\Omb )}\leq C \left( \| f(\cdot,u)\|_{L^{r}(\Om )}+\| f_B(\cdot,u)\|_{L^{r_B}(\p \Om )}\right),
\end{equation}
where $m=\min \left\{ r^*, \dfrac{Nr_{B}}{N-1}\right\}\, $, if $1\leq r<N$, or $m=\min \left\{ r, \dfrac{Nr_{B}}{N-1}\right\}$, if $r\geq N$ and $\nu =1-\frac{N}{m}.$\\
 
Besides, 
\begin{equation*}
\|u\|_{L^{\infty}(\p\Om)}\le \|u\|_{C(\Omb)}=\|u\|_{L^{\infty}(\Om)}.   
\end{equation*}  
\end{thm}

\begin{proof}
Let $u\in H^{1}(\Om )$ be a weak solution to \eqref{E0.1}. Then $u\in L^{q}(\Om )\cap L^{q}(\p \Om )$ for all $q<\infty $ (see \cite[Theorem 3.1]{Marino_Winkert_NonlAnal_2019}).\\
Next, we use elliptic regularity theory. Since Hölder's inequality, 
\begin{equation*}
f(\cdot,u)\in L^{q}(\Om ), \qq{for every} 1<q<r ,
\end{equation*}
and  
\begin{equation*}
f_{B}(x,u)\in L^{q_{B}}(\p \Om ), \qq{for every} 1<q_B<r_B ,
\end{equation*}
and by elliptic regularity (see Theorem \eqref{Th:RegLineal}),  $ u\in W^{1,m}(\Om )$ for $m=\min \left\{ q^*, \dfrac{Nq_{B}}{N-1}\right\}\, \, $  whenever $1\leq q<N$, or $m=\min \left\{ q, \dfrac{Nq_{B}}{N-1}\right\}$, whenever $q\geq N.$ 

Thanks to $r>N/2$ and $r_B>N-1$, we can always choose 
\begin{equation*}
q\in(N/2,r), \quad q_B\in(N-1,r_B).
\end{equation*}
and then  $m>N$, so $u\in C^{\nu }(\Omb)$ for  $\nu  =1-\frac{N}{m}$.\\

Moreover, since $u\in C^{\nu }(\Omb)$, $a\in L^{r}(\Om)$ and $\tilde{f}\in C(\Omb )$,  using the H\"older inequality, then, the product $| a(\cdot)| |\tilde{f}(u(\cdot))| \in L^{r}(\Om)$. Hence, $f(\cdot , u(\cdot))\in L^{r}(\Om)$. \\

Similarly, if $u\in C^{\nu }(\Omb)$, $a_{B}\in L^{r_{B}}(\p\Om)$ and $\tilde{f_{B}}\in C(\p \Om )$, by H\"older inequality, then, the product $| a_{B}(\cdot)| |\tilde{f_{B}}(u(\cdot))| \in L^{r_{B}}(\p\Om)$. Hence, we can conclude that $f_{B}(\cdot , u(\cdot))\in L^{r_{B}}(\p\Om)$. \\

Then \eqref{estim:W1m} and \eqref{estim:C:nu} hold, ending the proof.
\end{proof}

\end{appendix}
\bibliographystyle{abbrv}
\bibliography{ref}
\end{document}